\documentclass{amsart}
\usepackage{amsfonts}
\usepackage{amsmath}
\usepackage{amssymb}
\usepackage{amscd}
\usepackage{MnSymbol}

\usepackage[all,cmtip]{xy}

\newtheorem{thm}{Theorem}[section]
\newtheorem{lm}[thm]{Lemma}
\newtheorem{cor}[thm]{Corollary}
\newtheorem{prop}[thm]{Proposition}
\newtheorem{exm}[thm]{Example}

\theoremstyle{definition}

\newcommand{\Hom}{\text{\rm Hom}}

\newcommand{\Cc}{\mathcal{C}}
\newcommand{\Nc}{\mathcal{N}}
\newcommand{\Fc}{\mathcal{F}}
\newcommand{\Tc}{\mathcal{T}}
\newcommand{\Sc}{\mathcal{S}}
\newcommand{\M}{\mathcal{M}}
\newcommand{\Z}{\mathbb{Z}}
\newcommand{\Q}{\mathbb{Q}}
\newcommand{\Pb}{\mathbb{P}}
\DeclareMathOperator{\Add}{{\rm Add}}

\DeclareMathOperator{\Ker}{\rm ker}
\DeclareMathOperator{\End}{\rm End}

\DeclareMathOperator{\im}{\rm im}

\subjclass[2000]{20K40 (20K20, 20K21)}
\keywords{self-small abelian group, slender group}
\thanks{This work is part of the project SVV-2020-260589}

\begin{document}
\title{Self-small products of abelian groups}
\author{Josef Dvo\v r\'ak}
\address{CTU in Prague, FEE, Department of mathematics, Technick\'a 2, 166 27 Prague 6 \&
MFF UK, Department of Algebra,  Sokolovsk\' a 83, 186 75 Praha 8, Czech Republic} 
\email{pepa.dvorak@post.cz}

\author{Jan \v Zemli\v cka}
\address{Department of Algebra, Charles University,
Faculty of Mathematics and Physics Sokolovsk\' a 83, 186 75 Praha 8, Czech Republic} 
\email{zemlicka@karlin.mff.cuni.cz}

\begin{abstract}
For abelian groups $A, B$, A is called $B$-small if the covariant functor $\Hom(A,-)$ commutes with all direct sums $B^{(\kappa)}$ and $A$ is self-small provided it is $A$-small. The paper characterizes self-small products applying developed closure properties of the classes of relatively small groups. As a consequence, self-small products of finitely generated abelian groups are described. 
\end{abstract}
\date{\today}
\maketitle


Research of modules whose covariant functor $\Hom(M,-)$ commutes with all direct sums, which is a condition providing a categorial generalization the notion of finitely generated module, started in 60's by the work of Hyman Bass \cite[p.54]{Bass68} and  Rudolf Rentschler  \cite{Ren69}. Such modules have appeared as a useful tool in diverse contexts and under various terms (small, $\Sigma$-compact, U-compact, dually slender) in ring theory, module theory and in the study of abelian groups. David M. Arnold and Charles E. Murley published their influential paper \cite{A-M}, which is dedicated to a particular case of the studied condition by narrowing it to commuting with direct sums of the tested module itself, in  1974.  Groups and modules satisfying this restricted condition are usually called self-small in literature. Many interesting results concerning self-small modules over unital rings in general have appeared later \cite{AB,BZ,CoMe,GNM,Mo19},  self-small abelian groups proving to be a particularly successful tool \cite{ABS, ABW,  Pur, Br03, B20, Dualities}.

The aim of this paper is to deepen the present knowledge about structure of self-small groups and about possibilities of testing abelian groups for self-smallness by adopting some ideas of the papers \cite{ABS,DZ21,KZ19} and extending several results of \cite{Dv15,Z08}. Namely, we deal with the notion of a relatively small abelian group (defined in \cite{ABS,GNM}, cf. also  relatively compact objects in \cite{KZ19}) which serves as a tool for characterization of those products of groups that are self-small.

Throughout the paper \emph{module} means a right module over an associative ring with unit and an abelian group is a module over the ring of integers. Note that we will use the term \emph{group} instead of abelian group frequently, as non-abelian groups are not considered here.
If $A$ and $B$ are two abelian groups, then $\Hom(A,B)$ denotes the abelian group of homomorphisms $A\to B$. A \emph{family of groups} means a discrete diagram in the category of abelian groups, so a family may contain more that one copy of a group. 
The set of all prime numbers is denoted by $\mathbb{P}$ and we identify cardinals with least ordinals of given cardinality. 

For non-explained terminology we refer to \cite{Fuchs1,FuchsII}.

\section{Relatively small groups}

Let $A$, $B$ be abelian groups and $\Nc$ a family of abelian groups. It is well-known (and easy to see) that the functor 
$\Hom(A,-)$ induces an injective homomorphism of abelian groups
\begin{equation*}
\Psi_\Nc: \bigoplus_{N\in\Nc} \Hom(A,N) \to \Hom(A,\bigoplus\Nc)
\end{equation*}
by the rule
$\Psi_\Nc((f_N)_N)=\sum_N f_N$ (cf. e.g. \cite[Lemma 1.3]{KZ19}).
Suppose, then, that $\Cc$ is a class of groups and $B$ is an abelian group.
We say that $A$ is \emph{$\Cc$-small} if $\Psi_\Nc$ is an isomorphism for each subfamily $\Nc$ of class $\Cc$ and 
 $A$ is said to be {\it $B$-small} provided it is a $\{B\}$-small group (cf. \cite{ABS,DZ21,GNM,KZ19}). It is clear that  $A$-small abelian groups $A$  are exactly  \emph{self-small} ones as defined in \cite{A-M}.

\begin{exm}\rm (1) Every finitely generated abelian group is small, so $B$-small for every group $B$. In, particular each finite group is self-small.

(2) Let $A$ and $B$ be two abelian groups such that $\Hom(A,B)=0$. Then it is easy to see that $A$ is $B$-small.

In particular, if $p,q\in\mathbb P$ are different primes, $A_p$ is an abelian $p$-group and $A_q$ is an abelian $q$-group, then $A_p$ is $A_q$-small and $\Z$-small.
\end{exm}

\begin{exm} \rm It is clear, $\Q$ and $\Q/\Z$ are  $\Q$-small groups but
neither $\Q$ nor  $\Q/\Z$ is $\Q/\Z$-small. Furthermore, $\Q$-small groups are precisely groups of finite torsion-free rank by \cite[Corollary 4.3.]{ABS}.
\end{exm}

We start with an elementary observation which translates the definition of a relative small group to an easily tested condition  (cf. \cite[Section 1]{Ren69}, \cite[Lemma 1.4(2)]{KZ19} and \cite[Theorem 1.6(2)]{DZ21}):

\begin{lm}\label{B-small} Let $A$ and $B$ be abelian groups and $\Cc$ a class of abelian groups. Then $A$ is $\Cc$-small if and only if for each  family $\Nc$ of groups contained in the class $\Cc$ and every  $f\in\Hom(A,\bigoplus\Nc)$ there exists a finite family $\Fc\subseteq \Nc$ such that $f(A)\subseteq \bigoplus\Fc$. 
In particular, $A$ is $B$-small if and only if for every index set $I$ and every  $f\in\Hom(A,B^{(I)})$ there exists a finite subset $F\subseteq I$ such that $f(A)\subseteq B^{(F)})$.
\end{lm}
\begin{proof} The argument of the proof is well known; 
if $\Psi_\Nc$ is onto and $f\in\Hom(A,\bigoplus\Nc)$, then there exist finitely many $f_i\in \Hom(A,N_i)$, $i=1,\dots,n$ such that $\Psi_\Nc(\oplus_{i}f_i)=f$, hence $f(A)\subseteq \bigoplus_{i=1}^nN_i$. On the other hand, if $f(A)\subseteq \bigoplus_{i=1}^nN_i\subseteq \bigoplus\Nc$, then $\Psi_\Nc(\oplus_{i}\pi_{N_i}f)=f$, where $\pi_{N_i}$ denotes the projection onto $i$-th component.
\end{proof}

The observation that the concept of relatively small groups is general enough if we consider  relative smallness over a set of groups (cf. general \cite[Lemma 2.1]{DZ21}) presents a first application of the previous lemma. To that end, for a class of groups define 
$$\Add(\Cc) = \{A \, | \, A \text{ is a direct sumand of } \bigoplus_{\alpha<\kappa} C_{\alpha} \text{ for some cardinal } \kappa \text{ and } C_{\alpha} \in \Cc\}$$

and by $\Add(A)$ denote $\Add(\{A\})$.

\begin{lm}\label{Cc-small} Let $A$ be an abelian group and $\Cc$ be a set of abelian groups.
Then the following conditions are equivalent:
\begin{enumerate}
\item $A$ is $\bigoplus\Cc$-small,
\item $A$ is $\Cc$-small,
\item $A$ is $\Add(\bigoplus\Cc)$-small.
\end{enumerate}
\end{lm}
\begin{proof} (1)$\Rightarrow$(3) Put $B=\bigoplus\Cc$, let  $\Nc$ be a family of groups contained in  $\Add(B)$, and  $f\in\Hom(A,\bigoplus\Nc)$. Then for each $N\in\Nc$ there exists a cardinal $\kappa_N$ for which $N \subseteq B^{(\kappa_N)}$ ($N$ is also a direct summand of $B^{(\kappa_N)}$), and so $f(A)\subseteq \bigoplus\Nc\subseteq \bigoplus_{N\in\Nc}B^{(\kappa_N)}$. Since $A$ is $B$-compact, there exists finite family $\Fc\subseteq \Nc$ such that $f(A)\subseteq \bigoplus_{N\in\Fc}B^{(\kappa_N)}$ which implies that  $f(A)\subseteq \bigoplus\Fc$.

(3)$\Rightarrow$(2) It is obvious since $\Cc\subseteq\Add(\bigoplus\Cc)$. 

(2)$\Rightarrow$(1) As any group $B\in\Cc$ is a direct summand of $\bigoplus\Cc$, the same argument as in the implication (1)$\Rightarrow$(3) proves the assertion.
\end{proof}

Since $\Add(B)=\Add(B^{(\kappa)})$ for an arbitrary group $B$ and a nonzero cardinal $\kappa$, we obtain the following useful criterion:

\begin{cor}\label{power}   Let $A$ and $B$ be abelian groups and $\kappa$ a nonzero cardinal. Then  $A$ is $B$-small if and only if $A$ is $B^{(\kappa)}$-small.
\end{cor}

As a consequence, we can formulate a well-known closure property of the class of all self-small groups.

\begin{cor}\label{M^k} Let $\kappa$ be a cardinal and $A$ an abelian group.
Then $A^{(\kappa)}$ is self-small if and only if $A$ is self-small and $\kappa$ is finite. 
\end{cor}

Let us formulate a variant of the  assertion \cite[Theorem 4.1.]{ABS}, which generalize the classical criterion of self-small groups \cite[Proposition 1.1]{A-M} for the case of relatively small groups (cf. \cite[Lemma~3.3]{DZ21}). Recall that the family $(A_i\mid i<\omega)$ is said to be \emph{$\omega$-filtration} of a group $A$, if it is a chain of subgroups of $A$, i.e. $A_i\subseteq A_{i+1}$ for each $i<\omega$, with $A=\bigcup _{n<\omega}  A_n$.

\begin{prop}\label{notB-small} 
The following conditions are equivalent for abelian groups $A$ and $B$:
 \begin{enumerate}
\item $A$ is not $B$-small,
\item there exists a homomorphism $f\in \Hom(A,B^{(\omega)})$ such that $f(A)\nsubseteq B^{(n)}$ for all $n<\omega$,
\item there exists an $\omega$-filtration $(A_i\mid i<\omega)$ of $A$ such that for each $n<\omega$ there exists a nonzero $f_n\in \Hom(A,B)$ satisfying $f_n(A_n)=0$,
\item there exists an $\omega$-filtration $(A_i\mid i<\omega)$ of $A$ such that $\Hom(A/A_n,B)\ne 0$ for each $n<\omega$. 
\end{enumerate}
\end{prop}

\begin{proof} The proof works using similar arguments as in \cite[Proposition 1.1]{A-M}.

(1)$\Rightarrow$(2) By Lemma~\ref{B-small} there exists a set $I$ and $g\in \Hom(A,B^{(I)})$ such that $g(A)\nsubseteq B^{(F)}$ for any finite $F\subset I$. Then we can construct by induction a sequence of finite sets $I_n\subset I$ such that $I_0=\emptyset$, $|I_n\setminus I_{n-1}|=1$ and $\Ker\pi_{I_{n-1}}g\supsetneq \Ker\pi_{I_n}g$ for all $n<\omega$ where $\pi_{I_n}\in\Hom(B^{(I)}, B^{(I_n)})$ denotes the natural projection. If we put $I_\omega=\bigcup_{i<\omega}I_i$, then $\pi_{I_\omega}g\in \Hom(A,B^{(I_\omega)})$ represents the desired homomorphism.

(2)$\Rightarrow$(3) Let $f\in \Hom(A,B^{(\omega)})$ satisfy the condition (2) and define $A_n=f^{-1}(B^{(n,\omega)})$ where $B^{(n,\omega)}=\{b\in B^{(\omega)}\mid \pi_i(b)=0\ \forall i\le n\}$ for natural projections $\pi_i: B^{(\omega)}\to B$ onto the $i$-th coordinate. Then $A=\bigcup _{i<\omega}  A_i$ and for each $i<\omega$ there exist $n_i>i$ such that $f_i=\pi_{n_i}f\ne 0$ with $f_i(A_i)=0$.

(3)$\Rightarrow$(4) It is enough to observe that any nonzero $f_n\in \Hom(A,B)$ satisfying $f_n(A_n)=0$ can be factorized through  the natural projection $\pi: A\to A/A_n$, i.e. there exists nonzero $\tilde{f}_n\in\Hom(A/A_n,B) $ for which $\tilde{f}_n \pi= f$.

(4)$\Rightarrow$(1) Let $f_i\in \Hom(A/A_i,B)$ denote a nonzero homomorphism
and define a homomorphism $f\in \Hom(A,B^{\omega})$ by the rule 
$\pi_i(f(a))=f_i(\pi_{A_i}(a))=f_i(a+A_i)$ for each $a\in A$ and $i<\omega$.
Then $f(A)\subseteq B^{(\omega)}$ since for each $a\in A$ there exists $n$ such that $a\in A_i$ for all $i\ge n$, hence $f\in \Hom(A,B^{(\omega)})$. On the other hand, $f(A)\nsubseteq B^{(n)}$ for any $n<\omega$ as $\pi_nf\ne 0$, $i<\omega$. Thus $A$ is not $B$-small  by Lemma~\ref{B-small}.
\end{proof}

The previous assertion applied on $A=B$ allows us to reformulate \cite[Proposition 9]{Dv15}.
 
\begin{cor} \label{char}
The following conditions are equivalent for an abelian groups $A$:
 \begin{enumerate}
\item $A$ is not self-small,
\item there exists  an $\omega$-filtration $(A_i\mid i<\omega)$ of $A$ such that $\Hom(A/A_n,A)\ne 0$ for each $n<\omega$,
\item there exists  an $\omega$-filtration $(A_i\mid i<\omega)$ of $A$ such that for each $n<\omega$ there exists a nonzero $\varphi_n\in \End(A)$ satisfying $\varphi_n(A_n)=0$. 
\end{enumerate}
\end{cor}

\begin{exm} \rm
Put $P=\prod_{p \in \Pb}\Z_p$. Then  $\Hom(\Q,P)=0$ by \cite[Example 4]{Dv15}, hence $\Q$ is $P$-small. On the other hand, if we put 
$B= P/\bigoplus_{p \in \Pb}\mathbb{Z}_p$, then there exists
exists an $\omega$-filtration
$(B_i\mid i<\omega)$ of $B$ such that $\Hom(B/B_n,\mathbb Q)\ne 0$ for each $n$ by \cite[Example 3]{Dv15}. If we take preimages $A_n$ of all $B_n$ in canonical projection $P\to P/\bigoplus_{p \in \Pb}\mathbb{Z}_p$, then 
$(A_i\mid i<\omega)$ forms an $\omega$-filtration of $A$ satisfying
$\Hom(A/A_n,\mathbb Q)\cong\Hom(B/B_n,\mathbb Q)\ne 0$, hence $P$ is not $\Q$-small by Proposition~\ref{B-small} (equivalently, we  could use \cite[Corollary 4.3.]{ABS}).
\end{exm}

\section{Closure properties of relative smallness}

First, let us formulate several elementary relations between classes of relatively small groups.

\begin{lm}\label{factor}  Let $A$, $B$ and $C$ be abelian groups and $I$ be a set. Suppose that $A$ is $B$-small.
\begin{enumerate}
\item If $C$ is a subgroup of $A$, then $A/C$ is $B$-small.
\item If $C$ is embeddable into $B^I$, then $A$ is $C$-small.
\end{enumerate}
\end{lm}
\begin{proof} (1) Proving indirectly, we assume that $\overline{A}=A/C$ is not $B$-small.
 Then there exists  an $\omega$-filtration $(\overline{A}_i\mid i<\omega)$ of $\overline{A}$ for which $\Hom(\overline{A}/\overline{A}_n,B)\ne 0$ for all $n<\omega$ by Proposition~\ref{notB-small}. If we lift all the groups of the $\omega$-filtration of $\overline{A}$ to the $\omega$-filtration $(A_i\mid i<\omega)$ of $A$ satisfying the conditions $C\le A_n$ and $A_n/C=\overline{A}_n$ for each $n$, then $\Hom(A/A_n,B)\cong\Hom(\overline{A}/\overline{A}_n,B)\ne 0$, hence $A$ is not $B$-small by Proposition~\ref{notB-small}.

(2) We may suppose w.l.o.g. that $C\le B^I$. Assume $A$ is not $C$-small and consider the $\omega$-filtration $(A_i\mid i<\omega)$ of $A$ for which $\Hom(A/A_n,C)\ne 0$ provided by Proposition~\ref{notB-small}. Then we have $\Hom(A/A_n,B^I)\ne 0$ for each $n<\omega$ and since for each nonzero $f_n\in\Hom(A/A_n,B^I)$ there exists $i\in I$ such that $\pi_if_n\ne 0$, we conclude that $\Hom(A/A_n,B)\ne 0$ for every $n<\omega$, a contradiction.
\end{proof}

\begin{prop}\label{self-small-image}  Let $A$ be a self-small abelian group.
\begin{enumerate}
\item If $f\in\Hom(A,A^I)$ for an index set $I$, then $f(A)$ is self-small.
\item If $I\subseteq \End(A)$, then $A/\bigcap\{\ker \iota\mid \iota\in I\}$ is self-small.
\end{enumerate} 
\end{prop}
\begin{proof} (1) Since $A$ is $A$-small, $f(A)$ is $A$-small by Lemma~\ref{factor}(1). Thus $f(A)$ is $f(A)$-small by Lemma~\ref{factor}(2).

(2) If $\varphi: A\to A^I$ is defined by the rule $\pi_\iota\varphi=\iota$ for each $\iota\in I$, then $\ker\varphi =\bigcap\{\ker \iota\mid \iota\in I\}$, hence $A/\bigcap\{\ker \iota\mid \iota\in I\}\cong f(A)$ is self-small by (1) (cf. also \cite[Example 2.10]{DZ21}). 
\end{proof}

The next assertion describes closure properties concerning extensions.

\begin{prop}\label{sub-small}  Let $A$ and $C$ be abelian groups and $B\le C$.  
\begin{enumerate}
\item If both $B$ and $C/B$ are $A$-small, then $C$ is $A$-small.
\item If $A$ is $B$-small and $C/B$-small, then $A$ is $C$-small.
\end{enumerate}
\end{prop}
\begin{proof} Similarly as in Lemma~\ref{factor}, we will use throughout the whole proof the correspondence of relative nonsmallness and properties of $\omega$-filtrations given by Proposition~\ref{notB-small}.

(1)  Suppose that $(C_n\mid n<\omega)$ is an $\omega$-filtration of $C$. Then 
$(C_n\cap B\mid n<\omega)$ is  an $\omega$-filtration of $B$ 
and $(C_n+B/B\mid n<\omega)$ is an $\omega$-filtration of $C/B$.
Since $B$ and $C/B$ are $A$-small, there exists $n$ such that $f(B)=0$ whenever $f\in \Hom(C,A)$ satisfies $f(B\cap C_n)=0$,
and $\tilde{f}(C/B)=0$ whenever $\tilde{f}\in \Hom(C/B,A)$ satisfies $\tilde{f}(C_n+B/B)=0$.

Let $f\in \Hom(C,A)$ such that $f(C_n)=0$, then $f(B)=0$ as $f(C_n\cap B)=0$ and 
there exists $\tilde{f}\in \Hom(C/B,A)$ for which $\tilde{f}\pi_B=f$. Now, $\tilde{f}(C/B)=0$ since $\tilde{f}(C_n+B/B)=0$, hence $f=\tilde{f}\pi_B =0$.
We have proved that $C$ is an $A$-small group.

(2) Similarly, suppose that $(A_n\mid n<\omega)$ is an $\omega$-filtration of $A$. Since $A$ is $B$-small, there exists $n$ for which $\Hom(A/A_n,B)=0$, and so $\Hom(A/A_i,B)=0$ for each $i\ge n$. If $f\in \Hom(A/A_i,C)$ is nonzero, then $\pi_Bf\in \Hom(A/A_i,C/B)$ is nonzero because $f(A/A_i)\nsubseteq B$ for each $i\ge n$. Since there exists $k\ge n$ for which $\Hom(A/A_k,C/B)=0$ again, we get $\Hom(A/A_k,C)=0$, hence $A$ is $C$-small.
\end{proof}

\begin{exm}\label{Ex-ext}\rm
The implication of the previous claim cannot be reversed: 

(1) $\prod_{p\in\Pb}\Z_p$ is self-small by \cite[Theorem 2.5 and Example 2.7]{Z08}, but $\bigoplus_{p\in\Pb}\Z_p$ is not $\prod_{p\in\Pb}\Z_p$-small.

(2) Since $\Hom(\Q/\Z,\Q)=0$, the group $\Q/\Z$ is $\Q$-small, but  
$\Q/\Z$ is not $\Q/\Z$-small.
\end{exm}

\begin{lm}\label{sum}  Let $A$ be an abelian group and $\M$ a finite family of abelian groups.
\begin{enumerate}
\item If $N$ is $A$-small for each $N\in\M$, then $\bigoplus\M$ is $A$-small.
\item If $A$ is $N$-small for each $N\in\M$, then $A$ is $\bigoplus\M$-small.
\end{enumerate}
\end{lm}
\begin{proof}  Put $M=\bigoplus \M$. Both of the proofs proceed by induction on the cardinality of $\M$.

(1) If $|\M|\le 1$, there is nothing  to prove. 
Let the assertion hold true for $|\M|-1$ and put $M_N=\bigoplus \M\setminus\{N\}$ for arbitrary $N\in\M$.
Since  $M_N$ is $A$-small by the induction hypothesis,  $N$ is $A$-small by the hypothesis and $M/N\cong M_N$, we get that $M$ is $A$-small by Proposition~\ref{sub-small}(1).

(2) The same induction argument as in (1) shows $A$ is $M$-small by Lemma~\ref{sub-small}(2), since $A$ is $N$-small by the hypothesis and it is $M_N$-small for each $N\in\M$ by the induction hypothesis.
\end{proof}

As the main result of the section we describe which finite sums of relatively small abelian groups are again relatively small.

\begin{prop}\label{sum-small}  
Let $\M$ and $\Nc$ be finite families of abelian groups. 
The following conditions are equivalent:
 \begin{enumerate}
\item $\bigoplus\M$ is $\bigoplus\Nc$-small,
\item $M$ is $\bigoplus\Nc$-small for each $M\in \M$,
\item $\bigoplus\M$  is $N$-small for each $N\in \Nc$,
\item $M$ is $N$-small for each $M\in \M$ and $N\in \Nc$,
\item for each $M\in \M$, $N\in \Nc$, and $\omega$-filtration $(M_i\mid i<\omega)$ of $M$, there exist $i<\omega$ with $\Hom(M/M_i, N) = 0$.
\end{enumerate}
\end{prop}

\begin{proof}
(1)$\Rightarrow$(2) Put $F_M:=\bigoplus(\M\setminus\{M\})\le \bigoplus\M$ and once $(\bigoplus\M)/F_M\cong M$, the claim follows from  Lemma~\ref{factor}(1).

(1)$\Rightarrow$(3) Since $N\le \bigoplus\Nc$ the assertion is clear by Lemma~\ref{factor}(2).

(2)$\Rightarrow$(4), (3)$\Rightarrow$(4) It follows from  Lemma~\ref{factor} again.

The implication (4)$\Rightarrow$(3) is a consequence of  Lemma~\ref{sum}(1), while
the implication (3)$\Rightarrow$(1) is shown in Lemma~\ref{sum}(2).

(4)$\Leftrightarrow$(5) It is an immediate consequence of Proposition~\ref{notB-small}.
\end{proof}

As a consequence we reformulate \cite[Proposition 5]{Dv15}:

\begin{cor}\label{finite_sums} The following conditions are equivalent for a finite family of  abelian groups $\M$ and $M=\bigoplus \M$:
 \begin{enumerate}
\item $M$ is self-small,
\item $N_1$ is $N_2$-small for each $N_1,N_2\in \M$,
\item for every  $N_1,N_2\in \M$ and $\omega$-filtration $(M_i\mid i<\omega)$ of $N_1$ there exist $i<\omega$ with $\Hom(N_1/M_i, N_2) = 0$.
\end{enumerate}
\end{cor} 

\begin{exm} \rm Since $\Hom(\Q,\Z)=0$ and $\Q$ is self-small and $\Z$ is small so $\Z$-small and $\Q$-small, the group $\Z\oplus \Q$ 
is self-small by Corollary~\ref{finite_sums}. 
\end{exm}

\section{Self-small products}

We start the section by a criterion of self-smallness of a general  product (cf. \cite[Theorem 5.4]{DZ21}).

\begin{thm}\label{criterion}  Let $\M$ be a family of abelian groups and put $M=\prod\M$ and $S=\bigoplus\M$.
Then the following conditions are equivalent:
\begin{enumerate}
\item $M$ is self-small,
\item $M$ is $S$-small,
\item $M$ is $\bigoplus\Cc$-small for each countable family $\Cc\subseteq \M$.
\end{enumerate}
\end{thm}

\begin{proof} The implications (1)$\Rightarrow$(2)$\Rightarrow$(3) follow from Lemma~\ref{factor}(2), since $S$ is embeddable into $M$ and  $\bigoplus\Cc$ is embeddable into $S$.

(3)$\Rightarrow$(1) Proving indirectly, assume that $M$ is not self-small.
Then there exists  an $\omega$-filtration $(M_i\mid i<\omega)$ of $M$ for which $\Hom(M/M_n,M)\ne 0$ for all $n<\omega$ by Proposition~\ref{notB-small}.
Using the same argument as in the proof of Proposition~\ref{sub-small}(2), for each $n<\omega$ there exists $A_n\in\M$ such that $\Hom(M/M_n,A_n)\ne 0$. If we put $\Cc=\{A_i\mid i<\omega\}$, then all $A_i$'s are embeddable into $\bigoplus\Cc$, hence $\Hom(M/M_n,\bigoplus\Cc)\ne 0$ for each $n<\omega$, which implies that $M$ is not $\bigoplus\Cc$-small by Proposition~\ref{notB-small}.
\end{proof}

As $A^\kappa$ is $A^{(\kappa)}$-small if and only if it is $A$-small by Corollary~\ref{power} we obtain the following consequence of Theorem~\ref{criterion}.

\begin{cor}\label{powerAC} Let  $A$ be an abelian group and $I$ a set. Then $A^I$ is self-small if and only if it is $A$-small.
\end{cor}

\begin{exm}\label{prod Z_p} {\rm (1) $\Q^{\omega}$ is not self-small, since it is an infinitely generated $\Q$-vector space, hence it is not $\Q$-small.

(2) We have recalled in Example~\ref{Ex-ext} that $\prod_{p\in\Pb}\Z_p$ is self-small, so it is $\bigoplus_{p\in\Pb}\Z_p$-small group by Theorem~\ref{criterion}. }
\end{exm}

Let us denote by $T_A=\bigoplus_{p\in\Pb}A_{(p)}$ the torsion part of an abelian group $A$ where $A_{(p)}$ denotes the $p$-component of the torsion part.

\begin{lm}\label{torsion-prod}  
Let $p\in \Pb$, $P$ be a nonzero $p$-group, $R$ a nonzero torsion group, $\Tc$  a family of finite torsion groups, and $\kappa$ be a cardinal. Then:
 \begin{enumerate}
\item $\Z_p^\kappa$ is $P$-small if and only if $\kappa$ is finite,
\item $\Z^\kappa$ is $R$-small if and only if $\kappa$ is finite,
\item if $\prod\Tc$ is $P$-small, then $\{T\in\Tc\mid T_{(p)}\ne 0\}$ is finite,
\item  if $\prod\Tc$ is $R$-small, then $\{T\in\Tc\mid T_{(p)}\ne 0\}$ is finite for each $p\in \Pb$ satisfying $R_{(p)}\ne 0$.
\end{enumerate}
\end{lm}

\begin{proof} (1) If $\kappa$ is finite, then $\Z_p^\kappa$ is finite, and so $P$-small (it is, in fact, small). If $\kappa$ is infinite, then $\Z_p^\kappa$ is an infinitely generated vector space over $\Z_p$. Hence infinite direct sum of groups $\Z_p$, which is not $\Z_p$-small, so it is not $P$-small by Lemma~\ref{factor}(2), since there exists $Q \leq P$ with $Q \simeq \Z_p$.

(2) It is enough to prove the direct implication. Suppose that $\kappa$ is infinite. Since there exists $p\in\Pb$ such that $R_{(p)}\ne0$ and $\Z^\kappa/ (p\Z^\kappa)\cong \Z_p^\kappa$ is not $R_{(p)}$-small by (1). Then  $\Z^\kappa$ is not $R$-small by Lemma~\ref{factor}(1).

(3) Put $\Tc_p= \{T_{(p)}\mid T\in\Tc, T_{(p)}\ne 0\}$ and $\Sc=\{pS \mid S\in\Tc_p\}$ and suppose that $\kappa=|\Tc_p|=|\{T\in\Tc\mid T_{(p)}\ne 0\}|$ is infinite. Then $(\prod \Tc_p)/\prod \Sc\cong \Z_p^\kappa$ which is not $P$-small by (1), and so $\prod \Tc_p$ is not $P$-small by Lemma~\ref{factor}(1). Now $\prod \Tc$ is not $P$-small by Lemma~\ref{factor}(1) again, as $\prod \Tc_p$ is a direct summand of $\prod \Tc$.

(4) It follows from (3) and Lemma~\ref{factor}(2).
\end{proof}

\begin{lm}\label{p-groups}  
Let $A_p$ be a finite $p$-group for each $p\in \Pb$. Then $\prod_{p\in \Pb}A_p$ is self-small.
\end{lm}

\begin{proof} Repeating the argument of \cite[Lemma 1.7]{Z08} (cf. also Example~\ref{prod Z_p}(2)) we can see that  
$A=\prod_{p\in \Pb}A_p/\bigoplus_{p\in \Pb}A_p$ is divisible, since $A_q$ is $p$-divisible for every $p\in \Pb$ except for $p=q$. Hence $A$ is $p$-divisible for all primes $p$ and so it is divisible. If $f\in \Hom(\prod_{p\ne q}A_p, A_q)$ where $q\in \Pb$, then $\bigoplus_{p\ne q}A_p\subseteq \ker f$, hence $\im f$ is isomorphic to some factor of the divisible group $A$. Therefore $\im f$ is divisible, so  $\im f=0$. In consequence, $\Hom(\prod_{p\ne q}A_p, A_q)=0$  and the fact that $A_q$ is self-small for each $q\in\Pb$ implies that $\prod_{p\in \Pb}A_p$ is self-small by applying \cite[Proposition 1.6]{Z08}.
\end{proof}


Now we are ready to describe self-small products of finitely generated groups.

\begin{thm}\label{ss-product}  Let $\M$ be a family of nonzero finitely generated abelian groups such that at least one $N\in\M$ has nonzero torsion part and put $M=\prod\M$, $S=\bigoplus\M$ and $F=S/T_S$. 
	Then the following conditions are equivalent:
	\begin{enumerate}
		\item $M$ is self-small,
		\item $S$ is $\Z$-small and $S_{(p)}$-small for all $p\in\Pb$,
		\item $S_{(p)}$ is finite for each $p\in\Pb$ and $S/T_S$ is finitely generated
		\item there are only finitely many $A\in\M$ which are infinite and for each $p\in \Pb$ there are only finitely many $A\in\M$ with $A_{(p)}\ne 0$,
		\item the family $\{B\in\M\mid \Hom(B,A)\ne 0\}$ is finite for each $A\in\M$,
		\item there are only finitely many $A\in\M$ which are infinite and the family $\{B\in\M\mid \Hom(C, B)\ne 0\}$ is finite for each finite $C\in\M$,
		\item $M\cong F\oplus\prod_{p\in\Pb} M_p$ for a finitely generated free group $F$ and finite abelian $p$-groups $M_p$ for each $p\in\Pb$.
	\end{enumerate}
\end{thm} 

\begin{proof} It is well-known that for every finitely generated abelian group $A\in\M$ there exists a finitely generated free group $F_A$ and a finite torsion group $T_A$ such that $A\cong F_A\oplus T_A$. Put $F=\bigoplus_{A\in\M}F_A$ and $T=\bigoplus_{A\in\M}T_A$ and note that $S\cong F\oplus T$ where $F$ is a free abelian group and $T$ is the torsion part of $S$.
	Furthermore $M\cong \prod_{A\in\M}F_A \oplus \prod_{A\in\M}T_A$.
	
	(2)$\Leftrightarrow$(3)$\Leftrightarrow$(4) Note that $\Hom(T,\Z)=0$, hence $S$ is $\Z$-small if and only if $F$ is $\Z$-small. Thus $S$ is $\Z$-small if and only if $S/T_S\cong F$ is finitely generated which holds true if and only if there are only finitely many $A\in\M$ with nonzero $F_A$, i.e. which are infinite. 
	Furthermore, it is easy to see that $S$ is $S_{(p)}$-small  if and only if $S_{(p)}$ is finite  if and only if there exists only finitely many $A\in\M$ such that $A_{(p)}\ne 0$.
	
	(4)$\Rightarrow$(5) Take $A\in\M$. Then $\Hom(B,A)\ne 0$ if and only if $F_B\ne 0$ or there exists $p\in\Pb$ satisfying $(T_A)_{(p)}\ne 0\ne (T_B)_{(p)}$. Since $A$ is finitely generated, there exist $p_i\in\Pb$, $i=1,\dots,k$ such that $T_A=\bigoplus_{i=1}^k(T_A)_{(p_i)}$. In total, we get
	\[
	\{B\in\M\mid \Hom(B,A)\ne 0\}\subseteq \{B\mid F_B\ne 0\}\cup \bigcup_{i=1}^k\{B\mid (T_B)_{(p_i)}\ne 0\},
	\]
	where both sets on the right-hand side are finite.
	
	(5)$\Rightarrow$(4) Let $A\in\M$ be infinite, i.e  $F_A\ne 0$. If $B$ is infinite, then $\Hom(B,A)\ne 0$, hence there exist only finitely many infinite groups $B\in\M$. Similarly, if $A,B\in\M$ are such that $A_{(p)}\ne 0\ne B_{(p)}$, then $\Hom(B,A)\ne 0$, so for each $p\in\Pb$ there are only finitely many $B\in\M$ such that $B_{(p)}\ne 0$.

	(1)$\Rightarrow$(4) Since $M$ is self-small, it is $S$-small by Theorem~\ref{criterion}. Furthermore, $\prod_{A\in\M}F_A$ being a direct summand, hence a factor of $M$, it is $M$-small and in consequence $T$-small by Lemma~\ref{factor}(2), so $\prod_{A\in\M}F_A$ is finitely generated by Lemma~\ref{torsion-prod}(2). Therefore there exist only finitely many $A$ with $F_A\ne 0$. 
	Similarly, since $\prod_{A\in\M}T_A$ is $T$-small,
	there exist only finitely many $A\in\M$ such that $A_{(p)}=(T_A)_{(p)}\ne 0$ for each $p\in \Pb$ by Lemma~\ref{torsion-prod}(4).

	(3)$\Rightarrow$(7) Note that by (3) $F=\bigoplus_{A\in\M}F_A$ is finitely generated. Moreover,
	\[
	\prod_{A\in\M}T_A = \prod_{A\in\M}\bigoplus_{p\in\Pb}(T_A)_{(p)}\cong
	\prod_{A\in\M}\prod_{p\in\Pb}(T_A)_{(p)}\cong \prod_{p\in\Pb}\prod_{A\in\M}(T_A)_{(p)}\cong \prod_{p\in\Pb}\bigoplus_{A\in\M}(T_A)_{(p)},
	\] 
	because $T_A$ is finite for all $A\in\M$ and for each $p\in\Pb$ there exist only finitely many $A$ with $(T_A)_{(p)}\ne0$. Then $M_p=\bigoplus_{A\in\M}(T_A)_{(p)}$ is a finite $p$-group for all $p\in\Pb$ and
	$M\cong F\oplus \prod_{A\in\M}T_A\cong F\oplus \prod_{p\in\Pb} M_p$
	
	(7)$\Rightarrow$(1) By Theorem~\ref{criterion} it is enough to prove that $M$ is $F\oplus\bigoplus_{p\in\Pb} M_p$-small. Since $F$ is finitely generated, it is $F\oplus\bigoplus_{p\in\Pb} M_p$-small. As $\Hom(\prod_{p\in\Pb} M_p, F)=0$, it remains to show that $\prod_{p\in\Pb} M_p$ is $\bigoplus_{p\in\Pb} M_p$-small by Proposition~\ref{sum-small}, which holds true by Lemma~\ref{p-groups} and Theorem~\ref{criterion}.

	(5)$\Leftrightarrow$(6) The assertion concerning infinite groups follows from the equivalence of (4) and (5). The rest is a consequence of the fact that $\Hom(C,B)\ne 0$ if and only if $\Hom(B,C)\ne 0$ for each pair of finitely generated torsion abelian groups $B, C$.
\end{proof}

An uncountable cardinal $\kappa$ is \emph{measurable} if it admits a
$\kappa$-additive measure $\mu: \kappa \to \{0;1\}$ such that $\mu(\kappa) = 1$ and
$\mu(x) = 0$ for $x\in\kappa$.
A group $G$ is called \emph{slender}, if for any homomorphism
$f:\mathbb{Z}^{\omega} \to G$, $f(\mathbf{e}_i) = 0$ for almost all $i\in\omega$,
where $\mathbf{e}_i$ denotes the element of $Z^{\omega}$ with
$\pi_j(\mathbf{e}_i) = \delta_{i,j}$. 
Recall that $\mathbb{Z}$ is slender by \cite[Theorem 94.2]{FuchsII} and that  for a nonmeasurable cardinal
$\kappa$ we have $\Hom(Z^{\kappa}, Z) \cong
Z^{(\kappa)}$ by \cite[Corollary 94.5]{FuchsII} (cf. also \cite[Theorem 3.6]{ABS}).

\begin{lm}\label{Z^kappa} $\Z^\kappa$ is $\Z$-small for each cardinal $\kappa$.
\end{lm}

\begin{proof} For finite $\kappa$ there is nothing to prove, so let us suppose that $\kappa$ is infinite
and assume that $\Z^\kappa$ is not $\Z$-small. Then there exists a homomorphism $g\in\Hom(Z^{\kappa}, Z^{(\omega)})$ such that $\im \, g$ is infinitely generated by Proposition~\ref{notB-small}, hence 
$\im \, g$ is a free abelian group of infinite rank. Since $\im \, g\cong \Z^{(\omega)}$ is projective, $\mathbb{Z}^{(\omega)}$ is a direct summand of $\Z^\kappa$, i.e. there exists a group $A$ for which $\Z^\kappa\cong \Z^{(\omega)}\oplus A$.

First, assume that $\kappa=\omega$. Then $\Hom(\Z^{\omega},\Z)\cong \Z^{(\omega)}$ by \cite[Corollary 94.5]{FuchsII} as $\Z$ is slender by \cite[Theorem 94.2]{FuchsII}. Hence
\[
\Z^{(\omega)}\cong\Hom(\Z^{\omega},\Z)\cong\Hom(\Z^{(\omega)}\oplus A,\Z)\cong
\Hom(\Z^{(\omega)},\Z)\oplus\Hom(A,\Z)\cong \Z^{\omega}\oplus\Hom(A,\Z)
\]
which is impossible for cardinality reasons (i.e. $|\Z^{(\omega)}|<|\Z^{\omega}|$).

We have proved that $\Z^{\omega}$ is $\Z$-small, so $\kappa>\omega$. Let $\lambda\ge\kappa$ be a nonmeasurable cardinal (it exists, as for instance each singular cardinal is nonmesurable). Then $\Hom(Z^{\lambda}, \Z) \cong
Z^{(\lambda)}$ by \cite[Corollary 94.5]{FuchsII} and $\Z^\lambda\cong \Z^\lambda\oplus\Z^\kappa$ as $\lambda+\kappa=\lambda$, hence $\Z^\lambda\cong \Z^{(\omega)}\oplus B$ for $B=\Z^\lambda\oplus A$. We get
\[
\Z^{(\lambda)}\cong\Hom(\Z^{\lambda},\Z)\cong\Hom(\Z^{(\omega)}\oplus B,\Z)\cong \Z^{\omega}\oplus\Hom(A,\Z),
\]
which implies that $\Z^{\omega}$ is embeddable into $\Z^{(\lambda)}$, so it is an infinitely generated free group. This contradicts the fact that $\Z^{\omega}$ is $\Z$-small.
\end{proof}


\begin{exm} \rm
	Expressing Proposition~\ref{sub-small}(1) via the language of short exact sequences, we can say that relative smallness is transferred from the outer members to the middle one. The other direction, however, is more complicated: while Lemma~\ref{factor}(1) implies the transfer from the middle member to the right, the previous example shows that the transfer to the left does not occur generally: we have $\Z^{(\omega)} \hookrightarrow \Z^{\omega}$, but $\Z^{(\omega)}$ is not $\Z$-small.
\end{exm}

Using Corollary~\ref{powerAC} we can formulate an important consequence:

\begin{cor}\label{ss-Z^kappa} $\Z^\kappa$ is self-small for each cardinal $\kappa$.
\end{cor}

We finish the paper by a general criterion of self-small products of finitely generated groups.

\begin{thm}\label{ss-product0}  Let $\M$ be a family of nonzero finitely generated abelian groups and put $M=\prod\M$, $S=\bigoplus\M$ and $F=S/T_S$. 
Then the following conditions are equivalent:
\begin{enumerate}
\item $M$ is self-small,
\item either $T_S=0$, or $S_{(p)}$ is finite for each $p\in\Pb$ and $S/T_S$ is finitely generated
\item either all $A\in\M$ are free, or the family $\{B\in\M\mid \Hom(B,A)\ne 0\}$ is finite for each $A\in\M$,
\item either $M\cong \Z^\kappa$ for some cardinal $\kappa$, or $M\cong F\oplus\prod_{p\in\Pb} M_p$ for a finitely generated free group $F$ and finite abelian $p$-groups $M_p$ for each $p\in\Pb$.
\end{enumerate}
\end{thm} 

\begin{proof} The torsion part of $M$ is zero if and only all groups $A\in\M$ are free which means that $M\cong \Z^\kappa$ for some cardinal $\kappa$ is self-small by Corollary~\ref{ss-Z^kappa}. The case when the torsion part of $M$ is nonzero  follows directly  from Theorem~\ref{ss-product}.
\end{proof}

\begin{exm} \rm
	The finiteness condition in the previous theorem cannot be omitted without additional conditions, since e.g. the group $\Q \times \prod_{p \in \Pb}\Z_p$ is not self-small by \cite[Example 3]{Dv15}.
\end{exm}

\end{document}